\documentclass[a4paper,12pt]{article}
\usepackage{amssymb}
\usepackage{amsthm}
\usepackage{epsfig}
\usepackage{amsmath}
\usepackage{amssymb}
\usepackage{amscd}
\usepackage{url}
\usepackage{graphicx}
\usepackage{comment}
\newtheorem{thm}{\bf Theorem}
\newtheorem{cor}[thm]{\bf Corollary}

\newtheorem{conj}[thm]{\bf Conjecture}
\newtheorem{lem}[thm]{\bf Lemma}

\theoremstyle{remark}

\newcommand{\eps}{\varepsilon}

\newcommand{\brm}[1]{\operatorname{#1}}

\newcommand{\x}{{\bf x}}
\newcommand{\z}{{\bf z}}
\newcommand{\y}{{\bf y}}
\newcommand{\vv}{{\bf v}}

\title{The extremal function for disconnected minors}
\author{
Endre Cs\'oka 
\thanks{Mathematics Institute, University of Warwick. Email: {\tt csokaendre@gmail.com}.} 
\and
Irene Lo\thanks{Department of Industrial Engineering and Operations Research, Columbia University. Email: {\tt iyl2104@columbia.edu}.}\and
Sergey Norin\thanks{Department of Mathematics and Statistics, McGill University. Email: {\tt snorin@math.mcgill.ca}. Supported by an NSERC grant 418520.}\and
Hehui Wu\thanks{Department of Mathematics, University of Mississippi. Email: {\tt hwu@olemiss.edu}.} \and
Liana Yepremyan\thanks{School of Computer Science, McGill University. Email: {\tt liana.yepremyan@mail.mcgill.ca}. Supported by an NSERC grant 418520.}}

\begin{document}
\maketitle
\baselineskip 20pt
\begin{abstract}
For a graph $H$ let $c(H)$ denote the supremum of $|E(G)|/|V(G)|$ taken over all non-null graphs $G$ not containing $H$ as a minor. We show that $$c(H) \leq \frac{|V(H)|+\brm{comp}(H)}{2}-1,$$ when $H$ is a union of cycles, verifying conjectures of Reed and Wood~\cite{ReeWoo14}, and Harvey and Wood~\cite{HarWoo15}. 

We derive the above result from a theorem which allows us to find two vertex disjoint subgraphs with prescribed densities in a sufficiently dense graph, which might be of independent interest.
\end{abstract}

\section{Introduction}

A classical theorem of Erd\H{o}s and Gallai determines the minimum number of edges necessary to guarantee existence of a cycle of length at least $k$ in a graph with a given number of vertices. (All the graphs considered in this paper are simple.)

\begin{thm}[Erd\H{o}s and Gallai~\cite{ErdGal59}]\label{thm:ErdGal}
Let $k \geq 3$ be an integer and let $G$ be a graph with $n$ vertices and more than $(k-1)(n-1)/2$ edges. Then $G$ contains a cycle of length at least $k$. 
\end{thm}

One of the main results of this paper generalizes Theorem~\ref{thm:ErdGal} to a setting where, instead of a single cycle with prescribed minimum length, we are interested in obtaining a collection of vertex disjoint cycles. In the case when there are no restrictions on the lengths of cycles this problem was completely solved by Dirac and Justesen, who proved the following.

\begin{thm}[Dirac and Justesen~\cite{Just85}]\label{thm:Just}
Let $k \geq 2$ be an integer and let $G$ be a graph with $n\geq 3k$ vertices and more than $$\max\left\{(2k-1)(n-k), n - \frac{(3k-1)(3k-4)}{2} \right\}$$ edges. Then $G$ contains $k$ vertex disjoint cycles. 
\end{thm}

We phrase our extensions of the above results in the language of minors. A graph $H$ is \emph{a minor} of a graph $G$  if a graph isomorphic to $H$ can be obtained from a subgraph of $G$ by contracting edges. Mader~\cite{Mader68} proved that for every graph $H$ there exists a constant $c$ such that every graph on $n \geq 1$ vertices with at least $cn$ edges contains $H$ as a minor. A well-studied  extremal question in graph minor theory is determining the optimal value of $c$ for a given graph $H$. Denote by $v(G)$ and $e(G)$ the number of edges and vertices of a graph $G$, respectively.  Following Myers and Thomason~\cite{Myers2005}, for a  graph $H$ with $v(H) \geq 2$ we define  
$c(H)$ as the supremum of $e(G)/v(G)$ taken over all non-null graphs $G$ not containing $H$ as a minor. We refer to $c(H)$ as \emph{the extremal function of $H$}. 

The extremal function of complete graphs has been extensively studied. Dirac~\cite{Dirac64}, Mader~\cite{Mader68}, J{\o}rgensen~\cite{Jorgensen94}, and Song and Thomas~\cite{SonTho06} proved that $c(K_t)=t-2$ for $t \leq 5$, $t \leq 7$, $t=8$ and $t=9$, respectively. Thomason~\cite{Thomason01} determined the precise asymptotics of $c(K_t)$, proving
$$c(K_t)=(\alpha+o_t(1))t\sqrt{\log{t}},$$
for an explicit constant $\alpha=0.37...$. Myers and Thomason~\cite{Myers2005} have extended the results of~\cite{Thomason01} to general dense graphs, while  
Reed and Wood~\cite{ReeWoo14} and Harvey and Wood~\cite{HarWooAverage15}  have recently proved bounds on $c(H)$ for sparse graphs, with the main result of~\cite{ReeWoo14} implying  that $$c(H) \leq 3.895v(H)\sqrt{\ln d(H)},$$ 
for graphs $H$ with average degree $d(H) \geq d_0$ for some absolute constant $d_0$.  

The extremal function  was explicitly determined for several structured families of graphs. In particular, Chudnovsky, Reed and Seymour~\cite{ChuReeSey11} have shown that $c(K_{2,t})=(t+1)/2$ for $t \geq 2$, and Kostochka and Prince~\cite{KosPri10} proved that $c(K_{3,t})=t+3$ for $t \geq 6300$. 

We determine the extremal function of $2$-regular graphs in which every component has odd number of vertices.
Let $kH$ denote the disjoint union of $k$ copies of the graph $H$.
Note that Theorems~\ref{thm:ErdGal} and~\ref{thm:Just} imply that $c(C_k)=(k+1)/2$ for $k \geq 3$, and $c(kC_3)=2k-1$  for $k \geq 1$. For a general $2$-regular graph $H$  Reed and Wood~\cite{ReeWoo14} conjectured that $c(H) \leq 2v(H)/3 -1$, and Harvey and Wood~\cite[Conjecture 5.5]{HarWoo15} conjectured that $c(kC_r) \leq  (r+1)/2-1$ for $r \geq 3$, $k \geq 1$. Our first result verifies these conjectures.

\begin{thm}\label{thm:cycles} Let $H$ be a disjoint union of cycles. Then
\begin{equation}\label{e:cycledensity}
c(H) \leq \frac{v(H)+\brm{comp}(H)}{2}-1.
\end{equation}
\end{thm}
 
It is not hard to see and is shown in Section~\ref{sec:minors} that, if  every component of $H$ is odd, then the bound (\ref{e:cycledensity}) is tight.

Theorem~\ref{thm:cycles} follows immediately from Theorem~\ref{thm:ErdGal} and the following more general result, which we prove in Section~\ref{sec:minors}.

\begin{thm}\label{thm:union} Let $H$ be a disjoint union of $2$-connected graphs $H_1$,$H_2$,\ldots,$H_k$. Then
$$
c(H) \leq c(H_1)+c(H_2)+\ldots+c(H_k)+k-1.
$$
\end{thm}

Theorem~\ref{thm:union} additionally allows us  to determine the extremal function for the disjoint union of small complete minors.

\begin{cor}\label{cor:complete} $c(kK_t) = kt-k-1$ for $k\geq 1$ and $3 \leq t \leq 9$.
\end{cor}

Let us note that the restriction on connectivity of components of $H$ in Theorem~\ref{thm:union} is an artefact of the proof method, and  the following conjecture of Qian, which motivated our work, relaxes this restriction.

\begin{conj}[Qian~\cite{Qian}]\label{conj:main} Let $H$ be a disjoint union of non-null graphs $H_1$ and $H_2$ then
$$c(H) \leq c(H_1) + c(H_2)+1.$$ 
\end{conj}

We prove Theorem~\ref{thm:union} by showing that the graph $G$ with at least $(c(H_1)+c(H_2)+\ldots+c(H_k)+k-1)v(G)$ edges contains $k$ vertex disjoint subgraphs $G_1,\ldots, G_k$, such that $G_i$ is sufficiently dense to guarantee $H_i$ minor for every $1 \leq i \leq k$.
The bulk of the paper is occupied by the proof of the following technical theorem, which accomplishes that.

\begin{thm}\label{thm:main} Let $s,t \geq 1$ be real, and  
let $G$ be a non-null graph with $e(G) > (s+t+1)(v(G)-1)$. Then there exist vertex disjoint non-null subgraphs $G_1$ and $G_2$ of $G$ such that $e(G_1) > s(v(G_1)-1)$ and $e(G_2) > t(v(G_2)-1)$. 
\end{thm}

In Section~\ref{sec:minors} we derive Theorem~\ref{thm:union} from Theorem~\ref{thm:main}. We prove Theorem~\ref{thm:main} in Section~\ref{sec:proof}.

\section{Proof of Theorem~\ref{thm:union}}\label{sec:minors}

In this section we derive Theorem~\ref{thm:union} from  Theorem~\ref{thm:main} and prove a couple of easy related results.

Theorem~\ref{thm:main} is naturally applicable to the following variant of the extremal function. For a graph $H$ with $v(H) \geq 3$ define $c'(H)$ to be the supremum of $e(G)/(v(G)~-~1)$ taken over all graphs $G$ with $v(G) > 1$ not containing $H$ as a minor. Theorem~\ref{thm:main} implies the following variant of Conjecture~\ref{conj:main}.

\begin{cor}\label{cor:c1}
Let $H$ be a disjoint union of graphs $H_1$ and $H_2$ such that $v(H_1),v(H_2) \geq 3$. Then 
$$
c'(H) \leq c'(H_1)+c'(H_2)+1.
$$
\end{cor}
\begin{proof}
Let $s = c'(H_1)$ and $t=c'(H_2)$. Clearly $s,t \geq 1$. Let $G$ be a non-null graph such that  $e(G) > (s+t+1)(v(G)-1)$. Let $G_1$ and $G_2$ be the subgraphs of $G$ satisfying the conclusion of Theorem~\ref{thm:main}. Then $G_i$ contains $H_i$ as a minor for $i=1,2$. Therefore $G$ contains $H$ as a minor, as desired.
\end{proof}

We derive Theorem~\ref{thm:union} from Corollary~\ref{cor:c1} using the following observation.

\begin{lem}\label{lem:c1}
Let $H$ be a $2$-connected graph then $c'(H)=c(H)$.
\end{lem}
\begin{proof} Let $c=c(H)$.
Clearly $c'(H) \geq c$. Suppose for a contradiction that $c'(H)>c$, and there exists a graph $G$ such that $e(G)>c(v(G)-1)$ and $G$ does not contain $H$ as a minor. Let the graph $G_k$ be obtained from $k$ disjoint copies of $G$ by gluing them together on a single vertex. (I.e. $G_k = G^1 \cup G^2 \ldots G^k$, where $G^i$ is isomorphic to $G$ for $1 \leq i \leq k$ and there exists $v \in V(G)$ such that $V(G^i) \cap V(G^j)=\{v\}$ for all $1 \leq  i < j \leq k$.) It is well known that if a graph contains a $2$-connected graph as a minor then one of its maximal two connected subgraphs also contains it. Thus $G_k$ does not contain $H$ as a minor. However,
for sufficiently large $k$ we have$$\frac{e(G_k)}{v(G_k)} = \frac{ke(G)}{k(v(G)-1)+1}= c+\frac{k(e(G) - c(v(G)-1))-c}{k(v(G)-1)+1}>c(H),$$
a contradiction.
\end{proof}

\noindent \emph{Proof of Theorem~\ref{thm:union}.}
By Corollary~\ref{cor:c1} and Lemma~\ref{lem:c1} we have
$$c(H) \leq c'(H) \leq \sum_{i=1}^kc'(H_i) + k-1 = \sum_{i=1}^kc(H_i) + k-1. \qquad~\qed$$

In the remainder of the section we discuss lower bounds on the extremal function. Let $\tau(H)$ denote \emph{the vertex cover number} of the graph $H$, that is the minimum size of the set $X \subseteq V(H)$ such that $H - X$ is edgeless.

\begin{lem}\label{lem:tau} $c(H) \geq \tau(H)-1$  for every graph $H$. 
\end{lem}
\begin{proof}
Let $t=\tau(H)-1$, and  let $\bar{K}_{t,n-t}$ denote the graph on $n \geq t$ vertices obtained from the complete bipartite graph $K_{t,n-t}$ by making the $t$ vertices in the first part of the bipartition pairwise adjacent.
Then $\tau(G) \leq t$ for every minor $G$ of  $\bar{K}_{t,n-t}$. Therefore $H$ is not a minor $\bar{K}_{t,n-t}$, and 
$$\frac{e(\bar{K}_{t,n-t})}{v(\bar{K}_{t,n-t})}= \frac{nt-t(t+1)/2}{n} \to  t,$$
as $n \to \infty.$
\end{proof}

The following corollary follows immediately from Lemma~\ref{lem:tau} and implies that the bound in Theorem~\ref{thm:cycles} is tight whenever all components of $H$ are odd cycles, as claimed in the introduction.

\begin{cor} For every $2$-regular graph $H$ with $\brm{odd}(H)$ odd components we have
$$c(H) \geq \frac{v(H)+\brm{odd}(H)}{2}-1.$$
\end{cor}

We finish this section by proving Corollary~\ref{cor:complete}.

\begin{proof}[Proof of Corollary~\ref{cor:complete}] By the results of ~\cite{Dirac64,Jorgensen94,Mader68,SonTho06} we have $c(K_t) = t-2$ for $3 \leq t \leq 9$. Therefore $c(kK_t) \leq kt-k-1$ by Theorem~\ref{thm:union}. On the other hand, $\tau(kK_t)=k\tau(K_t)=k(t-1)$. Thus  $c(kK_t) \geq kt-k-1$ by Lemma~\ref{lem:tau}.
\end{proof}

\section{Proof of Theorem~\ref{thm:main}}\label{sec:proof}

We prove Theorem~\ref{thm:main} by first constructing a fractional solution and then rounding it in two stages.  

Let $n=v(G)$, and assume $V(G)=[n]:=\{1,2,\ldots,n\}$ for simplicity. Let $S^{G}:=[0,1]^{V(G)}$. We will use bold letters for elements of $S^{G}$ and denote components of a vector $\x \in S^{G}$ by $x_1,x_2,\ldots,x_n$. For $r \in [0,1]$, we denote by ${\bf r}$ a constant vector $(r,r,\ldots, r) \in S^G$.
For $\x \in S^{G}$ let $e(\x)=\sum_{ij \in E(G)}x_ix_j$. 

Suppose that $x_i \in \{0,1\}$ for every $i \in V(G)$, and let $A = \{ i \in V(G) \: | \: x_i =1\}$ and $B=V(G)-A$. If $e(\x) > {\bf s} \cdot \x - s$, $e({\bf 1} - \x) > {\bf t} \cdot ({\bf 1} - \x) - t$, ${\bf{x}}\neq {\bf 1}$ and $\x \neq {\bf 0}$, then the subgraphs $G_1$ and $G_2$ of $G$ induced by $A$ and $B$, respectively, satisfy the conditions of the theorem. 

The above observation motivates to consider the following functions. Let  $$f(\x)=e(\x)-\left( {\bf s}+{\bf \frac{1}{2}}\right) \cdot \x,$$ and let
$$g(\x)=e({\bf 1} -\x)-\left({\bf t}+{\bf \frac{1}{2}}\right) \cdot ({\bf 1}- \x).$$ We say that $\x \in S^G$ is \emph{balanced} if 
\begin{equation}\label{e:fbalanced}
f(\x) > -\frac{(s+\frac{1}{2})^2}{s+t+1},
\end{equation}
\begin{equation}\label{e:gbalanced}
g(\x)>-\frac{(t+\frac{1}{2})^2}{s+t+1}, 
\end{equation}
\begin{equation}\label{e:xlarge}
\|\x\|_1 \geq s+1, \; \mathrm{and}
\end{equation}
 \begin{equation}\label{e:xsmall}
\|{\bf 1}- \x\|_1 \geq t+1.
\end{equation}

\vskip 5pt
\noindent {\bf Claim 1:} There exists a balanced $ \x \in S^G$.
\begin{proof} Let $\x \equiv ( s+ \frac{1}{2})/(s+t+1)$. Note that $v(G) \geq 2(s+t+1)$, as $v(G)(v(G)-1)/2 \geq e(G) > (s+t+1)(v(G)-1)$. Therefore $$\|\x\|_1 = \frac{s+\frac{1}{2}}{s+t+1}v(G) \geq 2s+1 \geq s+1, $$ and (\ref{e:xlarge}) holds for $\x$. Further,
\begin{align*}f(\x) &=\left(\frac{s+\frac{1}{2}}{s+t+1} \right)^2e(G) -\left(s+\frac{1}{2}\right)\frac{s+\frac{1}{2}}{s+t+1} n \\ &= \left(\frac{s+\frac{1}{2}}{s+t+1} \right)^2 \left(e(G) - (s+t+1)n\right)\\ &>-\frac{(s+\frac{1}{2})^2}{s+t+1}, \end{align*}
implying (\ref{e:fbalanced}).	
The inequalities (\ref{e:gbalanced}) and (\ref{e:xsmall}) hold by symmetry.
\end{proof}	

For $\x \in S^G$ let $\brm{fr}(\x)=\{i \in [n] \:|\: 0< x_i <1 \}$ denote the set of vertices corresponding to the non-integral values of ${\bf x}$. 
	
\vskip 5pt
\noindent {\bf Claim 2:} Let a balanced $\x \in S^G$ be chosen so that $|\brm{fr}(\x)|$ is minimum. Then $\brm{fr}(\x)$ is a clique in $G$.
\begin{proof} Suppose for a contradiction that there exist $i,j \in \brm{fr}(\x)$ such that $ij \not \in E(G)$. Then $f(\x)$ and $g(\x)$ are linear as functions of $x_i$ and $x_j$. That is, there exists linear functions $\delta_f(\vv),\delta_g(\vv)$, such that $f(\x+\vv)=f(\x)+\delta_f(\vv)$ and $g(\x+\vv)=g(\x)+\delta_g(\vv)$ for every $\vv=(v_1,\ldots,v_n)$ satisfying $v_k=0$ for every $k\not \in \{i,j\}$. Therefore there exists a vector $\vv \not \equiv 0$ as above, such that $\delta_f(\vv) \geq 0$ and $\delta_g(\vv) \geq 0$. Let $\eps$ be chosen maximum so that  $0 \leq \x + \eps\vv \leq 1$. Then inequalities (\ref{e:fbalanced})and  (\ref{e:gbalanced})  hold for $\x+\eps\vv$ by the choice of $\vv$. 

Suppose that
$\|\x+\eps\vv\|_1 < s+1$. Then there exists $0< \eps' < \eps$  such that 
$\|\x'\|_1 = s+1$, where $\x' =\x+\eps'\vv$.  Therefore \begin{align*}
-\frac{(s+\frac{1}{2})^2}{s+t+1} <f(\x+\eps'\vv) \leq \frac{(s+1)^2}{2} - \left(s+\frac{1}{2}\right)(s+1).
\end{align*}
The above  implies 
$$\left(\frac{	1}{2}-\frac{1}{s+t+1}\right)\left(s+\frac{1}{2}\right)^2< \frac{1}{8},$$
which is clearly contradictory for $s,t \geq 1$.
Thus (\ref{e:xlarge}) (and, symmetrically, (\ref{e:xsmall})) holds for $\x+\eps\vv$. It follows that $\x+\eps\vv$ is balanced, contradicting the choice of $\x$.
 \end{proof}

Let $\y$ be balanced such that $C:=\brm{fr}(\y)$ is a clique. As we can no longer continue to modify $f(\y)$ and $g(\y)$ linearly as in Claim 2, we adjust them as follows. Let $A=\{i \in [n] \: | \: x_i=1\}$,  $B=\{i \in [n] \: | \: x_i=0\}$, $a=|A|$, $b=|B|$ and  $c=|C|$.
Let $q = \sum_{i \in C}y_i$, and let $r=\lfloor q \rfloor$. 
For $x \in S^G$, let $$\bar{f}(\x)= r\sum_{i \in C}x_i - \frac{r(r+1)}{2}-\sum_{\{i,j\} \subseteq C}x_ix_j + e(\x) -{\bf s} \cdot \x,$$
and let 
\begin{align*}
\bar{g}(\x)&= (c-r-1)\sum_{i \in C}(1-x_i) - \frac{(c-r)(c-r-1)}{2} \\&-\sum_{\{i,j\} \subseteq C}(1-x_i)(1-x_j) + e({\bf 1} -\x) -{\bf t} \cdot ({\bf 1} -\x).
\end{align*}

\vskip 5pt
\noindent {\bf Claim 3:} Let $\x \in S^G$ be such that $x_i \in \{0,1\}$ for $i \in C$. Then $\bar{f}(\x) \leq e(\x) -{\bf s} \cdot \x$, and $\bar{g}(\x) \leq e({\bf 1} -\x) -{\bf t} \cdot ({\bf 1} -\x)$.
\begin{proof}
To verify the first inequality it suffices to show that  $$r\sum_{i \in C}x_i - \frac{r(r+1)}{2}-\sum_{\{i,j\} \subseteq C}x_ix_j \leq 0,$$
for every $\x \in \{0,1\}^C$. Let $p=\sum_{i\in C}x_i$. We have
\begin{align*}
r&\sum_{i \in C}x_i - \frac{r(r+1)}{2}-\sum_{\{i,j\} \subseteq C}x_ix_j \\&= rp - \frac{r(r+1)}{2} - \frac{p(p-1)}{2}= \frac{p-r -(p-r)^2}{2} \leq 0,
\end{align*}
as desired. The inequality $\bar{g}(\x) \leq e({\bf 1} -\x) -{\bf t} \cdot ({\bf 1} -\x)$ follows analogously.
\end{proof}

By Claim 3 it suffices to find $\x \in \{0,1\}^{[n]}$ such that $\bar{f}(\x)> -s$, $\bar{g}(\x) > -t$, ${\bf{x}}\neq {\bf 1}$ and $\x \neq {\bf 0}$. We start by estimating  $\bar{f}(\y)$ and $\bar{g}(\y)$. 

\vskip 5pt
\noindent {\bf Claim 4:} We have 
\begin{equation}\label{e:fbary}
\bar{f}(\y) > \frac{a}{2}+\frac{q^2}{2c}-\frac{(s+\frac{1}{2})^2}{s+t+1}
\end{equation}
and
\begin{equation}\label{e:gbary}
\bar{g}(\y) > \frac{b}{2}+\frac{(c-q)^2}{2c}-\frac{(t+\frac{1}{2})^2}{s+t+1}
\end{equation}

\begin{proof} It suffices to prove (\ref{e:fbary}), as (\ref{e:gbary}) is symmetric. We have
\begin{align*}
\bar{f}(\y) - f(\y) &= \frac{1}{2}(q+a) + rq -  \frac{r(r+1)}{2}- \sum_{\{i,j\} \subseteq C}y_iy_j  \\
&= \frac{1}{2}(q+a) + rq -  \frac{r(r+1)}{2}- \frac{q^2}{2}+\frac{1}{2}\sum_{i \in C}y^2_i\\
&\geq  \frac{1}{2}(q+a) + rq -  \frac{r(r+1)}{2}- \frac{q^2}{2}+\frac{q^2}{2c} \\ 
&=\frac{1}{2}(q+a)-\frac{r}{2}-\frac{(q-r)^2}{2}+\frac{q^2}{2c}\\
&\geq  \frac{1}{2}(q+a)-\frac{q}{2}+\frac{q^2}{2c} =\frac{a}{2}+\frac{q^2}{2c}.
\end{align*}
As $\y$ is balanced, (\ref{e:fbary}) follows. 
\end{proof}

Note that Claim 4 implies that $\bar{f}(\y) > -s$ and $\bar{g}(\y) > -t$.

We assume now that 
\begin{equation}\label{e:qbounds}
r \leq 2s \qquad \mathrm{and} \qquad c-r-1 \leq 2t
\end{equation}
The other cases are easier, as we will exploit the fact that the complete subgraph $G_1$ of $G$ on more than $ 2s$ vertices satisfies the theorem requirements.

The proof of the next claim is analogous to that of Claim 2 and we omit it.
\vskip 5pt
\noindent {\bf Claim 5:} There exists $\z \in S^G$  such that $\bar{f}(\z) \geq \bar{f}(\y)$, $\bar{f}(\z) \geq \bar{f}(\y)$, $z_i=y_i$ for every $i \in V(G)-C$, $\|\z\|_1 >1$, $\|{\bf 1} - \z\|_1 >1$  and
$|\brm{fr}(\z)| \leq 1$.

\vskip 10pt Consider a vector $\z$ that satisfies Claim 5. Let $i \in C$ be a vertex such that $z_j \in \{0,1\}$ for every $j \in V(G) - \{i\}$. We suppose without loss of generality that $z_i \leq \frac{1}{2}$, as the case $z_i \geq \frac{1}{2}$ is analogous due to symmetry between $\z$ and ${\bf 1} - \z$.  Let $\z^*$ be obtained from $\z$ by setting $z^*_i=0$. Then $\z^* \neq {\bf 1}$, $\z^* \neq {\bf 0}$, and,
as noted above, it suffices to show that $\bar{f}(\z^*) > -s$ and $\bar{g}(\z^*) > -t$.
We do this in the next two claims.

\vskip 5pt \noindent {\bf Claim 6:} $\bar{f}(\z^*) > -s$. 

\begin{proof} Let $x = z_i$ for brevity. We have
$\bar{f}(\z^*) \geq  \bar{f}(\z) - (r+a-s)x$.  Recall that $\y$ is balanced, and $\|y\|_1 \leq r+a+1$. Therefore by (\ref{e:xlarge}) we have $s \leq r+a$, and using (\ref{e:fbary}) we have
\begin{align*}
\bar{f}(\z^*) &> \frac{a}{2}+\frac{q^2}{2c}-\frac{(s+\frac{1}{2})^2}{s+t+1} - (r+a-s)x \\ &\geq\frac{s-q}{2}+\frac{q^2}{2c}-\frac{(s+\frac{1}{2})^2}{s+t+1},
\end{align*}
as $x \leq \frac{1}{2}$, $r \leq q$. By (\ref{e:qbounds}), it suffices to show 
\begin{equation*}\label{e:c6case2}
\frac{3}{2}s-\frac{(s+\frac{1}{2})^2}{s+t+1}\geq \frac{q}{2} - \frac{q^2}{2(q+2t+1)}.
\end{equation*}
As the right side increases with $q$ for fixed $s$ and $t$, it suffices to verify this inequality when $q=2s+1$. In this case we have
\begin{align*}
\frac{3}{2}s&-\frac{(s+\frac{1}{2})^2}{s+t+1} = \frac{2s^2 + 6st +2s-1}{4(s+t+1)} \\ &\geq
\frac{2s + 4st +2t +1}{4(s+t+1)} = \frac{2s+1}{2} - \frac{(2s+1)^2}{2(2s+2t+2)}.
\end{align*}
as desired.
 \end{proof}
 
\vskip 5pt \noindent {\bf Claim 7:} $\bar{g}(\z^*) > -t$. 
\begin{proof} 
To simplify the notation we prove the symmetric statement for $\bar{f}$ instead. That is, if $z_i \geq \frac{1}{2}$ and $\z^*$ is obtained from $\z$ by setting $z_i$ to $1$, we show that  $\bar{f}(\z^*) > -s$. Denote $1-z_i$ by $x$ for the duration of this claim. Then $\bar{f}(\z^*)  \geq \bar{f}(\z)+(r-s)x.$ If $r \geq s$ the claim follows directly from Claim 4, and so we assume $s \geq r$. Using (\ref{e:fbary}) and the inequality $s \leq r+a$, which was shown to hold in Claim 6,  we have 
\begin{align*}
\bar{f}(\z^*) &\geq \frac{a}{2}+\frac{q^2}{2c}-\frac{(s+\frac{1}{2})^2}{s+t+1} + (r-s)x \\ &\geq \frac{a+r-s}{2}+\frac{q^2}{2c}-\frac{(s+ \frac{1}{2})^2}{s+t+1} \\& \geq 
\frac{q^2}{2c}-\frac{(s+ \frac{1}{2})^2}{s+t+1} \geq -s,
\end{align*}
as desired. 
 \end{proof}

We have now proved the theorem in the case when (\ref{e:qbounds}) holds. Therefore without loss of generality we assume that $c-r-1 > 2t$.  We will need the following variant of Claims 2 and 5.

\vskip 5pt
\noindent {\bf Claim 8:} There exists $\z \in \{0,1\}^{V(G)}$  such that $\bar{f}(\z) \geq  \bar{f}(\y)$, $\sum_{i \in C}z_i \leq \lceil \sum_{i \in C}y_i \rceil$, and
$z_i=y_i$ for every $v \in V(G)-C$. 
\begin{proof}
The argument analogous to the proof of Claim 2, applied to the linear functions $\bar{f}$ and $-\sum_{i \in C} x_i$, instead of $f$ and $g$, implies existence of  $\z' \in S^G$  such that $\bar{f}(\z') \geq  \bar{f}(\y)$, $\sum_{i \in C}z'_i \leq  \sum_{i \in C}y_i $, 
$z_i=y_i$ for every $v \in V(G)-C$, and $|\brm{fr}(\z')| \leq 1$. 

Let $i \in C$ be such that $z'_j \in \{0,1\}$ for every $j \in V(C)-\{i\}$. Let $k = r +|\{j \in A \: |\: ij \in E(G)\}| -s$ be the coefficient of $z_i$ in $\bar{f}$ considered as a linear function of $z_i$. Let $\z$ be obtained from $\z'$ by setting $z_i = 1$ if $k \geq 0$, and by setting $z_i = 0$, otherwise. Then $\bar{f}(\z) \geq \bar{f}(\z')$, and $\z$ satisfies the claim. 
 \end{proof}

Finally, we  consider  a vector $\z$ that satisfies Claim 8, and let  $W=\{i\in C \:|\: z_i=0 \}$.
As $$\sum_{i \in C}z_i \leq\left \lceil \sum_{i \in C}y_i\right \rceil \leq r+1,$$ we have $|W| \geq c-r-1 > 2t$. Thus the subgraphs $G_1$ and $G_2$ of $G$ induced on $\{i\in V(G) \:|\: z_i=1 \}$ and $W$, respectively, satisfy the conditions of the theorem.
 
\section{Concluding remarks}

\subsubsection*{Improving Theorem~\ref{thm:main}.}

The following conjecture strengthening several aspects of Theorem~\ref{thm:main}, appears to be plausible and  implies Conjecture~\ref{conj:main}.

\begin{conj}\label{conj:partition}
Let $s,t \geq 0$ be real, and let $G$ be a non-null graph with $e(G) \geq  (s+t+1)v(G)$. Then there exist vertex disjoint non-null subgraphs $G_1$ and $G_2$ of $G$ such that $e(G_1) \geq sv(G_1)$, $e(G_2) \geq tv(G_2)$, and $V(G_1) \cup V(G_2)=V(G)$.
\end{conj}

Adjusting the parameters involved in the proof of Theorem~\ref{thm:main} one can prove a number of weakenings of Conjecture~\ref{conj:partition}. In particular, Wu using these methods proved the following.

\begin{thm}[Wu~\cite{WuPrivate}]\label{thm:Wu}
Conjecture~\ref{conj:partition} holds if $s=t$, or  $e(G) \geq  (s+t+\frac{3}{2})v(G)$.
\end{thm} 

Finally, let us note that a beautiful theorem of Stiebitz can be considered as a direct analogue of Conjecture~\ref{conj:partition} for minimum, rather than average, degrees.

\begin{thm}[Stiebitz~\cite{Stiebitz96}]\label{thm:Stiebitz}
Let $s,t \geq 0$ be integers, and let $G$ be a graph with minimum degree $s+t+1$. Then there exist vertex disjoint subgraphs $G_1$ and $G_2$ with $V(G_1) \cup V(G_2)=V(G)$ such that the minimum degree of $G_1$ is at least $s$ and  the minimum degree of $G_2$ is at least $t$.
\end{thm}

Unfortunately, we were unable to adapt the proof of Theorem~\ref{thm:Stiebitz} to Conjecture~\ref{conj:partition}.

\subsubsection*{Improving Theorem~\ref{thm:cycles}.}

The bound on the extremal function provided by Theorem~\ref{thm:cycles} is not tight when some, but not all, components of $H$ are even cycles.
A stronger conjecture below, which differs only slightly from~\cite[Conjecture 5.7]{HarWoo15}, if true would determine the extremal function for all $2$-regular graphs.

\begin{conj}\label{conj:chgeneral}
Let $H$ be a $2$-regular graph with $\brm{odd}(H)$ odd components, then
$$c(H)= \frac{v(H)+\brm{odd}(H)}{2}-1,$$
unless $H=C_{2l}$, in which case $c(H)=(2l-1)/2$, or $H=kC_4$, in which case $c(H)=2k -\frac{1}{2}$. 
\end{conj}

\subsubsection*{Asymptotic density.}

Let $\brm{ex}_m(n,H)$ denote the maximum number of edges in a graph on $n$ vertices not containing $H$ as a minor. Then $$c(H)=\sup_{n \geq 1}\left\{\frac{\brm{ex}_m(n,H)}{n}\right\}.$$ 
The asymptotic density of graphs not containing $H$ as a minor is determined by a different function 
$$c_{\infty}(H)=\limsup_{n \to \infty}\left\{\frac{\brm{ex}_m(n,H)}{n}\right\},$$ 
defined by Thomason in~\cite{Thomason08}. If $H$ is connected then $c(H)=c_\infty(H)$, however the equality does not necessarily hold for disconnected graphs which are the subject of this paper. Some of the more advanced tools in graph minor theory could be used to bound $c_{\infty}(H)$, and Kapadia and Norin~\cite{KapNorDensity} were able to establish the following asymptotic analogues of Conjectures~\ref{conj:main} and~\ref{conj:chgeneral}.

\begin{thm}\label{thm:asymptotic1}
Let $H$ be a disjoint union of non-null graphs $H_1$ and $H_2$ then
$$c_{\infty}(H) \leq c_{\infty}(H_1) + c_{\infty}(H_2)+1.$$ 
\end{thm}

\begin{thm}\label{thm:asymptotic2}
Let $H$ be a $2$-regular graph with $\brm{odd}(H)$ odd components, then
$$c_{\infty}(H)= \frac{v(H)+\brm{odd}(H)}{2}-1,$$
unless $H=C_{2l}$, in which case $c_{\infty}(H)=(2l-1)/2$, or $H=kC_4$, in which case $c_{\infty}(H)=2k -\frac{1}{2}$.
\end{thm}

\vskip 10pt
\noindent {\bf Acknowledgement.} This research was partially completed at a workshop held at the Bellairs Research Institute
in Barbados in April 2015. We thank the participants of the workshop  and Rohan Kapadia for  helpful discussions. We are especially grateful to Katherine Edwards, who contributed to the project, but did not want to be included as a coauthor.

\bibliographystyle{plain}
\bibliography{snorin}
\end{document}